\documentclass{article}
\usepackage{emsnewsletter}

\usepackage[english]{babel}
\usepackage{hyperref}       
\usepackage{url}            
\usepackage{booktabs}       
\usepackage{amsfonts}       
\usepackage{tablefootnote}
\usepackage{verbatim}
\usepackage{nicefrac}       
\usepackage{microtype}      
\usepackage{lipsum}
\usepackage{mathtools}

\usepackage{amsmath,amssymb,amsthm}
\usepackage{algorithm,algorithmic}
\usepackage{pifont}
\usepackage{cases}
\usepackage{subcaption,graphicx}
\usepackage{stackengine}    
\usepackage{wrapfig}
\usepackage{enumitem}
\usepackage{multirow}
\usepackage{xcolor}
\usepackage{array}

\newtheorem{assumption}[theorem]{Assumption}

\newcommand{\R}{\mathbb{R}}

\newcommand{\EE}{\mathbb{E}}

\newcommand{\cD}{{\mathcal{D}}}

\newcommand{\cL}{{\mathcal{L}}}

\newcommand{\cQ}{{\mathcal{Q}}}

\def\<#1,#2>{\langle #1,#2\rangle}

\usepackage[colorinlistoftodos,bordercolor=blue,backgroundcolor=blue!20,linecolor=blue,textsize=scriptsize]{todonotes}

\usepackage{xspace}

\newcommand{\algname}[1]{{\sf  #1}\xspace}

\author{\textit{Eduard Gorbunov}\\ Mohamed bin Zayed University of Artificial Intelligence, Abu Dhabi, UAE\\
Moscow Institute of Physics and Technology, Moscow, Russia\thanks{Part of the work was done when the author was at MIPT.}
}
\title{\Large{Unified analysis of SGD-type methods}}

\begin{document}
   
\maketitle

\begin{abstract}
    This note focuses on a simple approach to the unified analysis of \algname{SGD}-type methods from \cite{gorbunov2020unified} for strongly convex smooth optimization problems. The similarities in the analyses of different stochastic first-order methods are discussed along with the existing extensions of the framework. The limitations of the analysis and several alternative approaches are mentioned as well.
\end{abstract}



\section{Introduction}
Consider the unconstrained minimization problem 
\begin{equation}
    \min\limits_{x\in\R^d} f(x), \label{eq:main_problem}
\end{equation}
where the objective $f:\R^d \to \R$ is assumed to be $L$-smooth, i.e., $\forall x,y\in \R^d$
\begin{equation}
    \|\nabla f(x) - \nabla f(y)\| \leq L\|x- y\| \label{eq:Smoothness}
\end{equation}
and $\mu$-strongly convex, i.e., $\forall x,y\in \R^d$
\begin{equation}
    f(y) \geq f(x) + \langle \nabla f(x), y - x \rangle + \frac{\mu}{2}\|y - x\|^2, \label{eq:strong_convexity} 
\end{equation}
where $\langle a, b\rangle$ denotes the standard inner product in $\R^d$ and $\|a\| = \sqrt{\langle a, a \rangle}$ is $\ell_2$-norm. More precisely, the stochastic versions of \eqref{eq:main_problem} are considered, where function $f$ has an expectation form
\begin{equation}
    f(x) = \EE_\xi\left[f_\xi(x)\right] \label{eq:expectation}
\end{equation}
or, in particular, a finite-sum form
\begin{equation}
    f(x) = \frac{1}{n}\sum\limits_{i=1}^n f_i(x). \label{eq:finite_sum}
\end{equation}
Problems of the form \eqref{eq:expectation} and \eqref{eq:finite_sum} arise in various applications, especially in machine learning \cite{shalev2014understanding} and statistics \cite{spokoiny2012parametric}. Such problems are usually solved via stochastic first-order methods, e.g., Stochastic Gradient Descent (\algname{SGD}) \cite{robbins1951stochastic}:
\begin{equation}
    x^{k+1} = x^k - \gamma g^k, \label{eq:prox_SGD}
\end{equation}
where $g^k$ is an (conditionally) unbiased estimator of the full gradient
\begin{equation}
    \EE_k[g^k] = \nabla f(x^k), \label{eq:unbiasedness}
\end{equation}
where $\EE_k[\cdot]$ denotes the expectation w.r.t.\ the randomness coming from $k$-th iteration. 

There exists a large number of different analyses of \algname{SGD} under various assumptions; see the recent work \cite{garrigos2023handbook} summarizing state-of-the-art results on the convergence of \algname{SGD}. Moreover, since the original work introducing \algname{SGD} \cite{robbins1951stochastic}, many \algname{SGD}-like methods have been developed, including variance-reduced variants \cite{gower2020variance}, distributed variants with communication compression \cite{mishchenko2019distributed}, and coordinate-wise randomization \cite{nesterov2012efficiency,richtarik2014iteration}.

Despite the seeming differences between many versions of \algname{SGD}, they can be analyzed using a similar pattern. This observation is made in \cite{gorbunov2020unified}, where the authors proposed a unified framework for the analysis of \algname{SGD}-like methods in different setups. This work discusses this approach in detail and considers several special cases. Several extensions and alternative approaches are also mentioned.

\section{Unified Analysis}

The main building block of the unified analysis from \cite{gorbunov2020unified} is the following parametric assumption on the gradient estimate $g^k$ and the problem itself.

\begin{assumption}[Assumption 4.1 from \cite{gorbunov2020unified}]\label{as:param_assumption}
    Let $\{x^k\}_{k\geq 0}$ be the iterates produced by \algname{SGD} (Algorithm in \eqref{eq:prox_SGD}), where stochastic gradients are unbiased (i.e., satisfy \eqref{eq:unbiasedness}). Assume that there exist non-negative constants $A, B, C, D_1, D_2 \geq 0, \rho \in (0,1]$ and a (possibly) random non-negative sequence $\{\sigma_k^2\}_{k\geq 0}$ such that the following two relations hold
    \begin{eqnarray}
    \EE_k\left[\|g^k\|^2\right] &\leq& 2A\left(f(x^k) - f(x^*)\right) + B\sigma_k^2 + D_1, \label{eq:second_moment_bound}\\
    \EE_k\left[\sigma_{k+1}^2\right] &\leq& (1 - \rho)\sigma_k^2 + 2C\left(f(x^k) - f(x^*)\right) + D_2. \label{eq:sigma_k+1_bound}
    \end{eqnarray}
\end{assumption}

The above assumption is motivated by the analysis of different \algname{SGD}-type methods and can be derived for standard setups. The simplest example is Gradient Descent (\algname{GD}) applied to the minimization problem of $L$-smooth function $f$. Indeed, in this case, $g^k = \nabla f(x^k)$ and a standard property of $L$-smooth functions \cite{nesterov2018lectures} gives
\begin{equation*}
    \EE_k\left[\|g^k\|^2\right] = \|\nabla f(x^k)\|^2 \leq 2L\left(f(x^k) - f(x^*)\right),
\end{equation*}
meaning that Assumption~\ref{as:param_assumption} holds with parameters $A = L$, $B = 0$, $\sigma_k \equiv 0$, $D_1 = 0$, $\rho = 1$, $C = 0$, $D_2 = 0$. 

Another standard example is \algname{SGD} with the noise having bounded variance, i.e., $g^k = \nabla f_{\xi_k}(x^k)$, where $\xi_k$ encodes stochasticity, is sampled from some distribution $\cD$ independently from previous steps, and
\begin{equation*}
    \EE_{\xi \sim \cD}\left[\nabla f_{\xi}(x)\right] = \nabla f(x),\quad \EE_{\xi \sim \cD}\left[\|\nabla f_{\xi}(x) - \nabla f(x)\|^2\right] \leq \sigma^2.
\end{equation*}
Then, using variance decomposition, one can derive
\begin{eqnarray}
    \EE_k\left[\|g^k\|^2\right] &=& \EE_{\xi_k \sim \cD}\left[\|\nabla f_{\xi_k}(x^k)\|^2\right] \notag\\
    &=& \|\nabla f(x^k)\|^2 + \EE_{\xi_k \sim \cD}\left[\|\nabla f_{\xi_k}(x^k) - \nabla f(x^k)\|^2\right] \notag\\
    &=& 2L\left(f(x^k) - f(x^*)\right) + \sigma^2, \label{eq:second_moment_UBV}
\end{eqnarray}
meaning that Assumption~\ref{as:param_assumption} holds with parameters $A = L$, $B = 0$, $\sigma_k \equiv 0$, $D_1 = \sigma^2$, $\rho = 1$, $C = 0$, $D_2 = 0$.

There are much more examples of situations when Assumption~\ref{as:param_assumption} is satisfied; some of them are discussed in detail further. Typically, parameters $A$ and $C$ are related to some smoothness properties of the objective and sampling (in the context of finite-sum problems, one can consider sampling of mini-batches from given distribution, see \cite{gower2019sgd}). Sequence $\{\sigma_k^2\}_{k\geq 0}$ is related to the variance reduction process (in a broad sense), and $\rho$ can be seen as the rate of this process; these aspects will be clarified later. Finally, $D_1$ and $D_2$ usually stand for the noises not handled via variance reduction, and $B$ is some constant.

Under Assumption~\ref{as:param_assumption}, it is possible to derive the following general result on the convergence of \algname{SGD} (Algorithm in \eqref{eq:prox_SGD}).
\begin{theorem}[Theorem~4.1 from \cite{gorbunov2020unified}]\label{thm:main_theorem}
    Let $f$ be $\mu$-strongly convex and Assumption~\ref{as:param_assumption} be satisfied. Assume that
    \begin{equation}
        \gamma \leq \min\left\{\frac{1}{\mu}, \frac{1}{A + CM}\right\} \label{eq:stepsize_condition}
    \end{equation}
    for some constant $M > \nicefrac{B}{\rho}$. Then the iterates of \algname{SGD} (Algorithm in \eqref{eq:prox_SGD}) satisfy
    \begin{eqnarray}
        \EE\left[V_k\right] \leq \left(1 - \min\left\{\gamma\mu, \rho - \frac{B}{M}\right\}\right)^k\EE[V_0] + \frac{(D_1 + MD_2) \gamma^2}{\min\left\{\gamma\mu, \rho - \frac{B}{M}\right\}}, \label{eq:main_result}
    \end{eqnarray}
    where $V_k = \|x^k - x^*\|^2 + M\gamma^2 \sigma_k^2$.
\end{theorem}
\begin{proof}
    The update rule of \algname{SGD} implies
    \begin{eqnarray*}
        \|x^{k+1} - x^*\|^2 &=&\|x^k - x^*\|^2 - 2\gamma \langle x^k - x^*, g^k \rangle + \gamma^2 \|g^k\|^2.
    \end{eqnarray*}
    Next, one can take the conditional expectation $\EE_k[\cdot]$ from both sides of the above identity and use the unbiasedness of $g^k$:
    \begin{eqnarray*}
        \EE_k\left[\|x^{k+1} - x^*\|^2\right] &\overset{\eqref{eq:unbiasedness}}{=}&  \|x^k - x^*\|^2 - 2\gamma \langle x^k - x^*, \nabla f(x^k) \rangle\\
        &&+ \gamma^2 \EE_k\left[\|g^k\|^2\right].
    \end{eqnarray*}
    Strong convexity of $f$ gives the upper bound for the second term in the right-hand side (RHS), and the last term is upper bounded due to \eqref{eq:second_moment_bound}:
    \begin{eqnarray}
        \EE_k\left[\|x^{k+1} - x^*\|^2\right] &\overset{\eqref{eq:strong_convexity}}{\leq}&  (1 - \gamma \mu)\|x^k - x^*\|^2 \notag\\
        && - 2\gamma \left(f(x^k) - f(x^*)\right) + \gamma^2 \EE_k\left[\|g^k\|^2\right] \notag\\
        &\overset{\eqref{eq:sigma_k+1_bound}}{\leq}& (1 - \gamma \mu)\|x^k - x^*\|^2 + B\gamma^2\sigma_k^2 + D_1\gamma^2 \notag\\
        && - 2\gamma(1 - A\gamma) \left(f(x^k) - f(x^*)\right). \label{eq:technical_1}
    \end{eqnarray}
    Summing up the above inequality with $M\gamma^2$ multiple of \eqref{eq:sigma_k+1_bound}, one can obtain
    \begin{align*}
        \EE_k&\big[V_{k+1}\big] = \EE_k\left[\|x^{k+1} - x^*\|^2 + M\gamma^2 \sigma_{k+1}^2\right]\\
        &\overset{\eqref{eq:technical_1},\eqref{eq:sigma_k+1_bound}}{\leq} (1 - \gamma \mu)\|x^k - x^*\|^2 + B\gamma^2\sigma_k^2 + D_1\gamma^2 \notag\\
        & - 2\gamma(1 - A\gamma) \left(f(x^k) - f(x^*)\right)\\
        & + (1-\rho)M\gamma^2\sigma_k^2 + 2CM\gamma^2\left(f(x^k) - f(x^*)\right) + M\gamma^2 D_2\\
        &= (1-\gamma\mu)\|x^k - x^*\|^2 + \left(1 - \rho + \frac{B}{M}\right)M\gamma^2\sigma_k^2\\
        & - 2\gamma \left(1 - (A + CM)\gamma\right)\left(f(x^k) - f(x^*)\right)+ (D_1 + MD_2)\gamma^2\\
        &\overset{\eqref{eq:stepsize_condition}}{\leq} \left(1 - \min\left\{\gamma\mu, \rho - \frac{B}{M}\right\}\right)V_k + (D_1 + MD_2)\gamma^2,
    \end{align*}
    where in the last step it is used that $\|x^k - x^*\|^2 + M\gamma^2 \sigma_k^2 = V_k$. Taking the full expectation from the derived inequality, one gets
    \begin{eqnarray*}
        \EE\left[V_{k+1}\right] &\leq& \left(1 - \min\left\{\gamma\mu, \rho - \frac{B}{M}\right\}\right)\EE\left[V_k\right]\\
        &&+ (D_1 + MD_2)\gamma^2.
    \end{eqnarray*}
    Finally, the above inequality implies
    \begin{eqnarray*}
        \EE\left[V_{k+1}\right] &\leq& \left(1 - \min\left\{\gamma\mu, \rho - \frac{B}{M}\right\}\right)^{k+1}\EE\left[V_0\right]\\
        &&+ (D_1 + MD_2)\gamma^2\sum\limits_{t=0}^k\left(1 - \min\left\{\gamma\mu, \rho - \frac{B}{M}\right\}\right)^t\\
        &\leq& \left(1 - \min\left\{\gamma\mu, \rho - \frac{B}{M}\right\}\right)^{k+1}\EE\left[V_0\right]\\
        &&+ (D_1 + MD_2)\gamma^2\sum\limits_{t=0}^\infty \left(1 - \min\left\{\gamma\mu, \rho - \frac{B}{M}\right\}\right)^t\\
        &\leq& \left(1 - \min\left\{\gamma\mu, \rho - \frac{B}{M}\right\}\right)^{k+1}\EE\left[V_0\right]\\
        &&+ \frac{(D_1 + MD_2)\gamma^2}{\min\left\{\gamma\mu, \rho - \frac{B}{M}\right\}},
    \end{eqnarray*}
    which concludes the proof.
\end{proof}

Since $V_k \geq \|x^k - x^*\|^2$, the above result means that \algname{SGD} under Assumption~\ref{as:param_assumption} and quasi-strong monotonicity converges (in expectation) linearly to some neighborhood of the solution. The neighborhood size depends on the noises $D_1$ and $D_2$ and stepsize $\gamma$. This kind of behavior is classical for \algname{SGD}. When $D_1$ and $D_2$ are non-zero, one needs either to use mini-batching or decrease stepsize to achieve any predefined accuracy, e.g., see \cite{stich2019unified}.

The main property of Theorem~\ref{thm:main_theorem} is that despite its generality, it provides tight rates of convergence in special cases (up to numerical factors). The next section supports this claim with different examples.

\section{Special Cases}

\paragraph{\algname{SGD}: uniformly bounded variance case.} As it is shown in \eqref{eq:second_moment_UBV}, when $f$ is $L$-smooth and stochastic gradient has uniformly bounded variance, Assumption~\ref{as:param_assumption} holds with parameters $A = L$, $B = 0$, $\sigma_k \equiv 0$, $D_1 = \sigma^2$, $\rho = 1$, $C = 0$, $D_2 = 0$. Plugging these parameters in Theorem~\ref{thm:main_theorem}, one can get that for any $\gamma \leq \nicefrac{1}{L}$ and $k \geq 0$ the iterates of \algname{SGD} satisfy
\begin{equation}
    \EE\left[\|x^k - x^*\|^2\right] \leq \left(1 - \gamma \mu\right)^k \|x^0 - x^*\|^2 + \frac{\gamma \sigma^2}{\mu}, \label{eq:UBV_convergence}
\end{equation}
which matches the classical results in this setup \cite{bottou2018optimization}.

\paragraph{\algname{SGD}: expected smoothness.} Consider the situation when $g^k = \nabla f_{\xi_k}(x^k)$, where $\xi_k$ encodes stochasticity, is sampled from some distribution $\cD$ independently from previous steps, $\EE_{\xi\sim \cD}[\nabla f_{\xi}(x)] = \nabla f(x)$, and so-called expected smoothness \cite{gower2019sgd} is satisfied:
\begin{equation}
    \EE_{\xi\sim \cD}\left[\|\nabla f_{\xi}(x) - \nabla f_{\xi}(x^*)\|^2\right] \leq 2\cL\left(f(x) - f(x^*)\right), \label{eq:expected_smoothness}
\end{equation}
where $\cL > 0$ is called the expected smoothness constant. As it is shown in \cite{gower2019sgd}, this assumption is satisfied in many situations, e.g., in the finite-sum case \eqref{eq:finite_sum} with $f_i$ being convex and $L_i$-smooth, \eqref{eq:expected_smoothness} holds for a wide range of sampling strategies including standard uniform sampling, importance sampling, and different approaches for forming mini-batches. That is, $\cL$ depends both on the distribution $\cD$ and structural properties of $f$.

Moreover, expected smoothness implies Assumption~\ref{as:param_assumption}:
\begin{eqnarray*}
    \EE_k\left[\|g^k\|^2\right] &=& \EE_{\xi_k \sim \cD}\left[\|\nabla f_{\xi_k}(x^k) - \nabla f_{\xi_k}(x^*) + \nabla f_{\xi_k}(x^*)\|^2\right]\\
    &=& 2\EE_{\xi_k \sim \cD}\left[\|\nabla f_{\xi_k}(x^k) - \nabla f_{\xi_k}(x^*)\|^2\right]\\
    &&+ 2\EE_{\xi_k \sim \cD}\left[\|\nabla f_{\xi_k}(x^*)\|^2\right]\\
    &\overset{\eqref{eq:expected_smoothness}}{\leq}& 4\cL\left(f(x^k) - f(x^*)\right) + 2\sigma_*^2,
\end{eqnarray*}
where $\sigma_*^2 = \EE_{\xi \sim \cD}\|\nabla f_{\xi}(x^*)\|^2$ is the variance at the solution. That is, Assumption~\ref{as:param_assumption} holds with parameters $A = 2\cL$, $B = 0$, $\sigma_k \equiv 0$, $D_1 = 2\sigma_*^2$, $\rho = 1$, $C = 0$, $D_2 = 0$. Plugging these parameters in Theorem~\ref{thm:main_theorem}, one can get that for any $\gamma \leq \nicefrac{1}{2\cL}$ and $k \geq 0$ the iterates of \algname{SGD} satisfy
\begin{equation}
    \EE\left[\|x^k - x^*\|^2\right] \leq \left(1 - \gamma \mu\right)^k \|x^0 - x^*\|^2 + \frac{2\gamma \sigma_*^2}{\mu}, \label{eq:exp_smoothness_convergence}
\end{equation}
which matches the best-known result in this setup \cite{gower2019sgd}.

\paragraph{\algname{SGD} in the interpolation regime.} Another popular setup for the analysis of \algname{SGD} is the so-called interpolation regime. In this regime, finite-sum problem \eqref{eq:finite_sum} is considered under the assumption that functions $f_1,\ldots, f_n$ have shared minimizer $x^*$, i.e., $f_i(x^*) = \min_{x\in\R^d}f_i(x)$ for all $i=1,\ldots,n$. In the context of machine learning applications, this means that the model is complicated enough to interpolate the training data perfectly. Such formulations are motivated by the training of over-parameterized models \cite{ma2018power, zhang2021understanding} that typically satisfy good properties for the convergence of optimization methods \cite{liu2022loss}.

In particular, when interpolation holds and $f_1, \ldots, f_n$ are smooth, then for \algname{SGD} with $g^k = \nabla f_{i_k}(x^k)$, where $i_k$ is sampled uniformly at random from $\{1,\ldots, n\}$ independently from previous steps, the following inequality holds:
\begin{eqnarray}
    \EE_{i_{k}}\left[\|\nabla f_{i_k}(x^k)\|^2\right] &\leq& \frac{1}{n}\sum\limits_{i=1}^n \|\nabla f_i(x^k)\|^2 \notag\\
    &\leq& \frac{1}{n}\sum\limits_{i=1}^n 2L\left(f_i(x^k) - f_i(x^*)\right) \notag\\
    &=& 2L\left(f(x^k) - f(x^*)\right). \label{eq:interploation_smoothness}
\end{eqnarray}
That is, in this case, Assumption~\ref{as:param_assumption} holds with parameters $A = L$, $B = 0$, $\sigma_k \equiv 0$, $D_1 = 0$, $\rho = 1$, $C = 0$, $D_2 = 0$. In view of Theorem~\ref{thm:main_theorem}, this means that the algorithm converges linearly for $\gamma \leq \nicefrac{1}{L}$:
\begin{equation}
    \EE\left[\|x^k - x^*\|^2\right] \leq \left(1 - \gamma \mu\right)^k \|x^0 - x^*\|^2. \label{eq:linear_convergence_SGD}
\end{equation}
In contrast to the standard case, when \algname{SGD} converges only to some neighborhood of the solution, in the interpolation regime, \algname{SGD} converges linearly to the exact solution asymptotically in expectation. In other words, due to the special structure of stochasticity, the method behaves similarly (neglecting the differences in the smoothness constants) to its deterministic counterpart -- \algname{GD} -- but has much cheaper iterations than \algname{GD}.

There exist modifications of the condition \eqref{eq:interploation_smoothness} called Relaxed Weak Growth Condition \eqref{eq:WGC} and Relaxed Strong Growth Condition \eqref{eq:SGC} \cite{schmidt2013fast, vaswani2019fast}:
\begin{align}
    \EE_{i_{k}}\left[\|\nabla f_{i_k}(x)\|^2\right] &= 2\rho L\left(f(x) - f(x^*)\right) + \sigma^2, \tag{R-WGC} \label{eq:WGC}\\
    \EE_{i_{k}}\left[\|\nabla f_{i_k}(x)\|^2\right] &= \rho \|\nabla f(x)\|^2 + \sigma^2, \tag{R-SGC} \label{eq:SGC}
\end{align}
where $L > 0$, $\rho \geq 0$ and $\sigma^2\geq 0$ are some parameters. When $\sigma^2 = 0$, both of them imply that $\nabla f_i(x^*) = 0$ for all $i \in \{1, \ldots, n\}$, meaning that $x^*$ is a shared minimizer for all $f_i$. Moreover, when $f$ is $L$-smooth, \eqref{eq:SGC} implies \eqref{eq:WGC}. Finally, \eqref{eq:WGC} implies Assumption~\ref{as:param_assumption} with parameters $A = \rho L$, $B = 0$, $\sigma_k \equiv 0$, $D_1 = \sigma^2$, $\rho = 1$, $C = 0$, $D_2 = 0$. Plugging these parameters in Theorem~\ref{thm:main_theorem}, one can get that for any $\gamma \leq \nicefrac{1}{\rho L}$ and $k \geq 0$ the iterates of \algname{SGD} satisfy \eqref{eq:UBV_convergence}.

\paragraph{Variance reduction.} In the standard (non-over-parameterized) finite-sum settings \eqref{eq:finite_sum}, \algname{SGD} with constant stepsize converges linearly only to some neighborhood of the solution. To fix this issue variance reduced methods were proposed \cite{schmidt2017minimizing, johnson2013accelerating, defazio2014saga}, see also the recent survey \cite{gower2020variance}. These methods use special perturbations of the stochastic gradients that do not change the mean but reduce the variance during the work of the method. 

How to choose these shifts/perturbations? To answer this question, consider \algname{SGD} with the standard uniform sampling: $g^k = \nabla f_{i_k}(x^k)$, where $i_k$ is sampled uniformly at random from $\{1,\ldots, n\}$ independently from previous steps. It is easy to show that for $f_i$ being convex and $L_i$-smooth, $i \in\{1,\ldots, n\}$ expected smoothness \eqref{eq:expected_smoothness} holds with $\cL = L_{\max} = \max_{i\in [n]}L_i$ \cite{gower2019sgd}. In this case, \eqref{eq:exp_smoothness_convergence} implies that \algname{SGD} converges to the neighborhood proportional to the variance at the optimum $\sigma_*^2 = \frac{1}{n}\sum_{i=1}^n \|\nabla f_i(x^*)\|^2$. Moreover, even if the method eventually occurs at the optimum, i.e., $x^k = x^*$, it will move away from this point since $\nabla f_i(x^*) \neq 0$ in general. Clearly, the gradients $\nabla f_i(x^*)$ play a crucial role in the convergence of \algname{SGD}, but, unfortunately, one cannot use them in the method.

Nevertheless, one can theoretically study the following non-implementable method called \algname{SGD-star} \cite{gorbunov2020unified, hanzely2019one} that uses the update rule \eqref{eq:prox_SGD} with
\begin{equation*}
    g^k = \nabla f_{i_{k}}(x^k) - \nabla f_{i_k}(x^*).
\end{equation*}
Since $\nabla f(x^*)$, $\EE_k[g^k] = \nabla f(x^k)$ and
\begin{eqnarray*}
    \EE_{i_k}\left[\|g^k\|^2\right] &=& \EE_{i_k}\left[\|\nabla f_{i_k}(x^k) - \nabla f_{i_k}(x^*)\|^2\right]\\
    &=& \frac{1}{n}\sum\limits_{i=1}^n\|\nabla f_i(x^k) - \nabla f(x^*)\|^2\\
    &\leq& \frac{2L_{\max}}{n}\sum\limits_{i=1}^nV_{f_{i}}(x^k, x^*)\\
    &=& 2L_{\max}\left(f(x^k) - f(x^*)\right),
\end{eqnarray*}
where $V_{\psi}(x, y) = \psi(x) - \psi(y) - \langle \nabla \psi(y), x - y \rangle$ is a Bregman divergence of function $\psi$ \cite{bregman1967relaxation, nemirovskij1983problem}. The above derivation implies that Assumption~\ref{as:param_assumption} holds in this case with parameters $A = L_{\max}$, $B = 0$, $\sigma_k \equiv 0$, $D_1 = 0$, $\rho = 1$, $C = 0$, $D_2 = 0$. Therefore, according to Theorem~\ref{thm:main_theorem} \algname{SGD-star} with stepsize $\gamma \leq \nicefrac{1}{L_{\max}}$ converges linearly as in \eqref{eq:linear_convergence_SGD}. Unfortunately, as it is noticed before, gradients $\nabla f_i(x^*)$ are not known in advance, meaning that \algname{SGD-star} is not practical in general.

However, it is possible to learn ``optimal shifts''/control variates \cite{nelson1990control} $\nabla f_i(x^*)$ on the fly, and, in some sense, this is what variance-reduced methods do. As an example, consider Loopless Stochastic Variance-Reduced Gradient (\algname{LSVRG}) \cite{hofmann2015variance, kovalev2020don}. \algname{LSVRG} follows the scheme \eqref{eq:prox_SGD} with
\begin{eqnarray}
    g^k &=& \nabla f_{i_k}(x^k) - \nabla f_{i_k}(w^k) + \nabla f(w^k), \label{eq:g^k_LSVRG}\\
    w^{k+1} &=& \begin{cases} x^k,& \text{with probability } p,\\ w^k,& \text{with probability } 1-p. \end{cases} \label{eq:w^K_LSVRG}
\end{eqnarray}
The gradient $\nabla f_{i_{k}}(w^k)$ serves as an approximation of $\nabla f(x^*)$ and $\nabla f(w^k)$ is added to the estimator to make it unbiased: $\EE_{k}[g^k] = \nabla f(x^k) - \nabla f(w^k) + \nabla f(w^k) = \nabla f(x^k)$. The full gradient is computed with (typically small) probability $p$ at each iteration at the reference point $w^k$, i.e., when $w^k$ is updated. When $w^{k+1} = w^k$ the algorithm requires only $2$ gradients computations: $\nabla f_{i_k}(x^k)$ and $\nabla f_{i_k}(w^k)$. Therefore, an expected oracle cost of one step equals $(2+n)p + 2(1-p) = np + 2$. For $p \sim \nicefrac{1}{n}$, the expected oracle cost of one step becomes $O(1)$, i.e., comparable to the oracle cost of one standard \algname{SGD} step.

The introduced shifts improve the upper bound for the second moment of the stochastic gradient:
\begin{eqnarray*}
    \EE_{k}\left[\|g^k\|^2\right] 
    &\leq& 2\EE_k \left[\|\nabla f_{i_k}(x^k) - \nabla f_{i_k}(x^*)\|^2\right]\\
    &&+ 2\EE_k \left[\|\nabla f_{i_k}(w^k) - \nabla f_{i_k}(x^*) - \nabla f(w^k)\|^2\right]\\
    &\leq& 2\EE_k \left[\|\nabla f_{i_k}(x^k) - \nabla f_{i_k}(x^*)\|^2\right]\\
    &&+ 2\EE_k \left[\|\nabla f_{i_k}(w^k) - \nabla f_{i_k}(x^*)\|^2\right],
\end{eqnarray*}
where the last step holds since the variance is upper-bounded by the second moment. Applying smoothness and convexity of $f_i(x)$, one can get
\begin{eqnarray}
    \EE_{k}\left[\|g^k\|^2\right] &=& \frac{2}{n}\sum\limits_{i=1}^n \|\nabla f_i(x^k) - \nabla f_i(x^*)\|^2 \notag\\
    &&+ 2\underbrace{\frac{1}{n}\sum\limits_{i=1}^n\|\nabla f_{i}(w^k) - \nabla f_{i}(x^*)\|^2}_{\sigma_k^2} \notag\\
    &\leq& \frac{4L_{\max}}{n}\sum\limits_{i=1}^n V_{f_i}(x^k, x^*) + 2\sigma_k^2 \label{eq:ncsjdbhvcsdbhj}\\
    &=& 4L_{\max}\left(f(x^k) - f(x^*)\right) + 2\sigma_k^2. \label{eq:LSVRG_second_moment}
\end{eqnarray}
This is the first example in this paper, when a non-trivial sequence of $\{\sigma_k\}_{k\geq 0}$ is used: it measures how the current shifts $\nabla f_i(w^k)$ are far from the ``ideal'' ones $\nabla f_i(x^*)$, which is aligned with the intuition provided earlier. Moreover, by definition of $\sigma_k$ and $w^{k+1}$
\begin{eqnarray}
    \EE_k\left[\sigma_{k+1}^2\right] &\overset{\eqref{eq:w^K_LSVRG}}{=}& (1-p)\sigma_k^2 + \frac{p}{n}\sum\limits_{i=1}^n \|\nabla f_i(x^k) - \nabla f_i(x^*)\|^2\notag\\
    &\leq& (1-p)\sigma_k^2 + \frac{2p L_{\max}}{n}\sum\limits_{i=1}^n V_{f_i}(x^k, x^*)\notag\\
    &=& (1-p)\sigma_k^2 + 2pL_{\max}\left(f(x^k) - f(x^*)\right). \label{eq:LSVRG_sigma_k}
\end{eqnarray}
One can interpret the above inequality as follows: with probability $1-p$, the approximations of the gradients at the solution do not change, and with probability $p$ the approximation changes, and its quality can be bounded as $2L_{\max}(f(x^k) - f(x^*))$. Since the change happens once every $\nicefrac{1}{p}$ steps, the quantity $p$ can be seen as the rate of the variance reduction process.

Now, everything is ready to get the rate of convergence of \algname{LSVRG}. Indeed, inequalities \eqref{eq:LSVRG_second_moment} and \eqref{eq:LSVRG_sigma_k} imply that Assumption~\ref{as:param_assumption} holds in this case with parameters $A = 2L_{\max}$, $B = 2$, $\sigma_k^2 = \frac{1}{n}\sum_{i=1}^n\|\nabla f_i(w^k) - \nabla f_i(x^*)\|^2$, $D_1 = 0$, $\rho = p$, $C = pL_{\max}$, $D_2 = 0$. Plugging these parameters in Theorem~\ref{thm:main_theorem} with $M = \nicefrac{4}{p}$, one can get that for any $\gamma \leq \nicefrac{1}{6L_{\max}}$ the iterates of \algname{LSVRG} satisfy
\begin{eqnarray}
    \EE\left[V_k\right] \leq \left(1 - \min\left\{\gamma\mu, \frac{p}{2}\right\}\right)^k V_0, \label{eq:LSVRG_rate}
\end{eqnarray}
where $V_k = \|x^k - x^*\|^2 + \nicefrac{4\gamma^2\sigma_k^2}{p}$. That is, unlike standard \algname{SGD}, \algname{LSVRG} converges linearly with constant stepsize. Moreover, formula \eqref{eq:LSVRG_rate} establishes linear convergence not only for $\EE[\|x^k - x^*\|^2]$ but also for $\EE[\sigma_k^2]$ supporting the discussed intuition about the role of sequence $\{\sigma_k^2\}_{k\geq 0}$.

It is possible to show that some other variance-reduced methods and their modifications fit the considered theoretical framework. For example, the celebrated \algname{SAGA} \cite{defazio2014saga} and recently proposed \algname{LSVRG} with arbitrary sampling \cite{qian2021lsvrg} can also be seen as special cases of the framework, see the details in \cite{gorbunov2020unified, danilova2022distributed}.

\paragraph{Distributed methods with compression.} Consider the distributed system, where $n$ workers can communicate with a central server. Assume that for all $i \in [n]$ worker $i$ stores the information about the function $f_i$, i.e., this worker can compute (stochastic) gradients of $f_i$. The goal of the workers is to minimize function $f$ having a finite-sum form \eqref{eq:finite_sum}.

In such setups, workers have to exchange some information about their functions (e.g., gradients), i.e., they need to communicate with a server. These communications are often a main bottleneck for distributed methods \cite{kairouz2021advances} (especially when the number of workers $n$ is large). Therefore, it is natural to apply some compression to the messages that workers need to send to the server.

There are many examples of compression operators, but for simplicity, this paper focuses on the unbiased compression operators only (see the examples of other classes in \cite{beznosikov2020biased, sahu2021rethinking}). Stochastic operator $\cQ : \R^d \to \R^d$ is called unbiased compression operator \cite{horvath2022stochastic} (or simply unbiased compressor) if there exist such $\omega \geq 0$ such that for all $x\in\R^d$ 
\begin{equation}
    \EE_{\cQ}[\cQ(x)] = x,\quad \EE_{\cQ}\left[\|\cQ(x) - x\|^2\right] \leq \omega \|x\|^2, \label{eq:quantization_def}
\end{equation}
where $\EE_{\cQ}$ denotes the expectation w.r.t.\ the randomness coming from $\cQ$. There are many examples of unbiased compressors; see \cite{beznosikov2020biased, horvath2021better}.

One of the basic distributed methods with unbiased compression is Compressed Distributed Gradient Descent (\algname{CDGD}) \cite{alistarh2017qsgd, khirirat2018distributed}: it follows the general scheme \eqref{eq:prox_SGD} with
\begin{equation}
    g^k = \frac{1}{n}\sum\limits_{i=1}^n \cQ(\nabla f_i(x^k)), \label{eq:CDGD}
\end{equation}
where compression operators $\cQ(\nabla f_1(x^k)), \ldots, \cQ(\nabla f_n(x^k))$ are applied on different workers independently from each other and previous iterations. Here for simplicity, the case when workers compute full gradients is considered, though it is possible to analyze the version with stochastic gradients similarly. If functions $f_1,\ldots, f_n$ are convex and $L_i$-smooth, $f$ is $L$-smooth, then using the independence of the compression operators, one can derive
\begin{eqnarray*}
    \EE_k\left[\|g^k\|^2\right] &=& \EE_k\left[\|g^k - \nabla f(x^k)\|^2\right] + \|\nabla f(x^k)\|^2\\
    &\leq& \EE_k\left[\left\|\frac{1}{n}\sum\limits_{i=1}^n (\cQ(\nabla f_i(x^k)) - \nabla f_i(x^k))\right\|^2\right]\\
    && + 2L\left(f(x^k) - f(x^*)\right)\\
    &=& \frac{1}{n^2}\sum\limits_{i=1}^n\EE_k\left[\|\cQ(\nabla f_i(x^k)) - \nabla f_i(x^k)\|^2\right]\\
    && + 2L\left(f(x^k) - f(x^*)\right)\\
    &\overset{\eqref{eq:quantization_def}}{\leq}& \frac{\omega}{n^2}\sum\limits_{i=1}^n \|\nabla f_i(x^k)\|^2 + 2L\left(f(x^k) - f(x^*)\right).
\end{eqnarray*}
Then, it holds that $\|\nabla f_i(x^k)\|^2 \leq 2\|\nabla f_i(x^k) - \nabla f_i(x^*)\|^2 + 2 \|\nabla f_i(x^*)\|^2$ and similarly to \eqref{eq:ncsjdbhvcsdbhj} one can upper-bound $\frac{1}{n}\sum_{i=1}^n \|\nabla f_i(x^k) - \nabla f_i(x^*)\|^2 \leq 2L_{\max}(f(x^k) - f(x^*))$. As a result, one can get
\begin{eqnarray*}
    \EE_k\left[\|g^k\|^2\right] &\leq& 2\left(L + \frac{2L_{\max}\omega}{n}\right)\left(f(x^k) - f(x^*)\right) + \frac{2\omega\zeta_*^2}{n},
\end{eqnarray*}
where $\zeta_*^2 = \frac{1}{n}\sum_{i=1}^n\|\nabla f(x^*)\|^2$, which is non-zero in general. Thus, Assumption~\ref{as:param_assumption} holds in this case with parameters $A = L + \nicefrac{2\omega L_{\max}}{n}$, $B = 0$, $\sigma_k \equiv 0$, $D_1 = \nicefrac{2\omega\zeta_*^2}{n}$, $\rho = 1$, $C = 0$, $D_2 = 0$ and, according to Theorem~\ref{thm:main_theorem} \algname{CDGD} with stepsize $\gamma \leq \nicefrac{1}{A}$ converges as
\begin{equation}
    \EE\left[\|x^k - x^*\|^2\right] \leq \left(1 - \gamma \mu\right)^k \|x^0 - x^*\|^2 + \frac{2\gamma\omega \zeta_*^2}{n\mu} \label{eq:CDGD_convergence}
\end{equation}
When functions on workers have different optima and $\cQ(\cdot)$ is not an identity operator ($\omega > 0$), $\zeta_*^2 > 0$ and the above formula gives linear convergence only to the neighborhood of the solution. The reason for this is that the variance of the compressed gradient is proportional to its squared norm, and when the norms of $\nabla f_i(x^*)$ are non-zero, the update of the method contains the noise coming from the compression even if the method is at the solution. From this perspective, \algname{CDGD} and \algname{SGD} applied to the finite-sum problems behave very similarly.

This similarity hints the solution to this issue. In particular, linear convergence to the exact optimum (asymptotically, in expectation) can be achieved for \algname{SGD} via the variance reduction. Thus, one needs to design a variance reduction mechanism to handle the variance coming from the compression. This was done in \cite{mishchenko2019distributed, horvath2022stochastic}, where the authors proposed the method called \algname{DIANA}: it has the form \eqref{eq:prox_SGD} with
\begin{eqnarray}
    g^k = \frac{1}{n}\sum\limits_{i=1}^n g_i^k, \label{eq:DIANA_g_k}
\end{eqnarray}
where
\begin{gather*}
    \Delta_i^k = \cQ(\nabla f_i(x^k) - h_i^k),\quad g_i^k = h_i^k + \Delta_i^k,\quad h_i^{k+1} = h_i^k + \alpha \Delta_i^k,
\end{gather*}
where $\alpha > 0$ and $h_i^0 = 0$ for all $i \in [n]$. One can show that \algname{DIANA} can be implemented using only compressed vector communications from workers to the server \cite{mishchenko2019distributed}. 

Now, the convergence of \algname{DIANA} will be discussed. If functions $f_1,\ldots, f_n$ are convex and $L_i$-smooth, $f$ is $L$-smooth, then using the independence of the compression operators, one can get for $h^k = \frac{1}{n}\sum_{i=1}^n h_i^k$
\begin{eqnarray*}
    \EE_k\left[\|g^k\|^2\right] &=& \EE_k\left[\|g^k - (\nabla f(x^k) - h^k)\|^2\right] + \|\nabla f(x^k) - h^k\|^2\\
    &\leq& \EE_k\left[\left\|\frac{1}{n}\sum\limits_{i=1}^n (\Delta_i^k - (\nabla f_i(x^k) - h_i^k))\right\|^2\right]\\
    && + 2\|\nabla f(x^k)\|^2 + 2\|h^k\|^2\\
    &\leq& \frac{1}{n^2}\sum\limits_{i=1}^n\EE_k\left[\|\Delta_i^k - (\nabla f_i(x^k) - h_i^k)\|^2\right]\\
    && + 4L\left(f(x^k) - f(x^*)\right) + \frac{2}{n}\sum\limits_{i=1}^n \|h_i^k - \nabla f_i(x^*)\|^2\\
    &\overset{\eqref{eq:quantization_def}}{\leq}& \frac{\omega}{n^2}\sum\limits_{i=1}^n \|\nabla f_i(x^k) - h_i^k\|^2 + 4L\left(f(x^k) - f(x^*)\right)\\
    &&\quad + \frac{2}{n}\sum\limits_{i=1}^n \|h_i^k - \nabla f_i(x^*)\|^2.
\end{eqnarray*}
Similarly to the analysis of \algname{CDGD}, one can use $\|\nabla f_i(x^k) - h_i^k\|^2 \leq 2\|\nabla f_i(x^k) - \nabla f_i(x^*)\|^2 + 2 \|h_i^k - \nabla f_i(x^*)\|^2$ and similarly to \eqref{eq:ncsjdbhvcsdbhj} one can upper-bound $\frac{1}{n}\sum_{i=1}^n \|\nabla f_i(x^k) - \nabla f_i(x^*)\|^2 \leq 2L_{\max}(f(x^k) - f(x^*))$. As a result, one can get
\begin{eqnarray*}
    \EE_k\left[\|g^k\|^2\right] \leq 4\left(L + \frac{L_{\max}\omega}{n}\right)\left(f(x^k) - f(x^*)\right) + 2\left(1 + \frac{\omega}{n}\right)\sigma_k^2,
\end{eqnarray*}
where $\sigma_k^2 = \frac{1}{n}\sum_{i=1}^n \|h_i^k - \nabla f_i(x^*)\|^2$. Moreover, when $\alpha \leq \nicefrac{1}{(1+\omega)}$ one can prove \cite{horvath2022stochastic}
\begin{eqnarray*}
    \EE_k\left[\sigma_{k+1}^2\right] \leq (1 - \alpha)\sigma_k^2 + 2\alpha L_{\max} \left(f(x^k) - f(x^*)\right).
\end{eqnarray*}
Thus, Assumption~\ref{as:param_assumption} holds in this case with parameters $A = 2L + \nicefrac{2\omega L_{\max}}{n}$, $B = 2 + \nicefrac{2\omega}{n}$, $\sigma_k^2 = \frac{1}{n}\sum_{i=1}^n \|h_i^k - \nabla f_i(x^*)\|^2$, $D_1 = 0$, $\rho = \alpha$, $C = \alpha L_{\max}$, $D_2 = 0$ and, according to Theorem~\ref{thm:main_theorem} \algname{DIANA} with stepsizes $\alpha = \nicefrac{1}{(1+\omega)}$, $\gamma \leq \nicefrac{1}{A + CM}$, $M = \nicefrac{2B}{\alpha}$ converges as
\begin{equation}
    \EE\left[V_k\right] \leq \left(1 - \gamma \mu\right)^k V_0, \label{eq:DIANA_convergence}
\end{equation}
where $V_k = \|x^k - x^*\|^2 + 4(1+\omega)(1+ \nicefrac{\omega}{n})\gamma^2 \sigma_k^2$. In contrast to \algname{CDGD}, \algname{DIANA} converges linearly to the exact solution (asymptotically in expectation). Therefore, one can see \algname{DIANA} as a variance-reduced version of \algname{CDGD}. One can also that versions of \algname{DIANA} with \algname{SGD}/\algname{SVRG}/\algname{SAGA} estimators fit the discussed analysis as special cases \cite{horvath2022stochastic, gorbunov2020unified}.

\paragraph{Coordinate-wise randomization.} Finally, consider stochastic methods with coordinate-wise randomization (one can check \cite{nesterov2012efficiency,richtarik2014iteration} for the detailed introduction). In its simplest form, Randomized Coordinate Descent (\algname{RCD}) has the form \eqref{eq:prox_SGD} with
\begin{equation}
    g^k = d \nabla_{i_k} f(x^k) e_{i_k} = d\langle \nabla f(x^k), e_{i_k} \rangle, \label{eq:RCD}
\end{equation}
where $i_k$ is sampled uniformly at random from $\{1,\ldots,d\}$, $\nabla_{i_k} f(x^k)$ denotes the $i_k$-th component of $\nabla f(x^k)$, and vectors $e_1,e_2,\ldots,e_d$ form a standard basis in $\R^d$. Direct calculations of expectations show that
\begin{equation*}
    \EE_k[g^k] = \frac{1}{d}\sum\limits_{i=1}^d d\nabla_{i} f(x^k) e_{i} = \nabla f(x^k)
\end{equation*}
and under convexity and $L$-smoothness of $f$
\begin{eqnarray*}
    \EE_k\left[\|g^k\|^2\right] &=& \frac{1}{d}\sum\limits_{i=1}^d d^2 (\nabla_i f(x^k))^2 = d\|\nabla f(x^k)\|^2\\
    &\leq& 2dL\left(f(x^k) - f(x^*)\right),
\end{eqnarray*}
meaning that Assumption~\ref{as:param_assumption} holds in this case with parameters $A = dL$, $B = 0$, $\sigma_k^2 \equiv 0$, $D_1 = 0$, $\rho = 1$, $C = 0$, $D_2 = 0$. Therefore, Therefore, according to Theorem~\ref{thm:main_theorem} \algname{RCD} with stepsize $\gamma \leq \nicefrac{1}{dL}$ converges linearly as in \eqref{eq:linear_convergence_SGD} when $f$ is $\mu$-strongly convex and $L$-smooth. Unfortunately, the described analysis does not cover the case of non-uniform sampling of coordinates.

\section{Extensions and Alternative Approaches}

\paragraph{Relaxed assumptions, composite optimization.} In fact, the original work \cite{gorbunov2020unified} considers a more general setup of composite optimization:
\begin{equation*}
    \min\limits_{x\in\R^d}\left\{f(x) + R(x)\right\},
\end{equation*}
where $f$ is a smooth function and $R(x)$ is a proper closed and convex function \cite{beck2017first}. Moreover, it is sufficient to assume that $f$ is $\mu$-quasi strongly convex \cite{necoara2019linear} and $f(x) - f(x^*) - \langle \nabla f(x^*), x - x^* \rangle \geq 0$ for all $x \in \R^d$.

\paragraph{Convex and non-convex cases.} Extension to the convex (non-strongly convex) composite case is given in \cite{khaled2020unified}. For the non-convex problems, the unified analyses of \algname{SGD}-like methods are proposed in \cite{khaled2020better} (without the support of variance-reduction) and in \cite{li2020unified} (with the support of variance-reduction and communication compression). The unified view on the analysis of \emph{optimal} (stochastic) first-order methods for non-convex optimization is given by \cite{danilova2022recent}.

\paragraph{Extensions for distributed methods.} There exist several modifications of the unified analysis discussed in this paper for different classes of distributed optimization algorithms. In particular, unified analysis of the methods with error feedback \cite{seide20141,stich2018sparsified} is proposed in \cite{gorbunov2020linearly} (for contractive biased compressors) and in \cite{danilova2022distributed} (for absolute biased compressors). In \cite{gorbunov2021local}, the authors extend the framework to the distributed methods with local steps. Modification for the decentralized methods is proposed in \cite{rajawat2020primal}.

\paragraph{Variational inequalities.} It is possible to modify the unified analysis to cover the methods for more general problems like variational inequalities \cite{facchinei2003finite,beznosikov2022smooth}. In particular, the extension for Stochastic Gradient Descent-Ascent-type methods is given in \cite{beznosikov2022stochastic} and for Stochastic Extragradient-type \cite{korpelevich1976extragradient,juditsky2011solving} methods it is done in \cite{gorbunov2022stochastic,beznosikov2022unified}.

\paragraph{Alternative approaches.} There are also several alternative approaches to the unification of the analysis of stochastic first-order methods for different setups. In \cite{stich2019unified}, the authors propose a simple approach of estimating the recurrences appearing in the analysis of  \algname{SGD} with different stepsize policies. The work \cite{ajalloeian2020convergence} gives an analysis of \algname{SGD} with biased estimators under quite general assumptions on the bias. The unified view on the analysis of the stochastic methods with and without variance reduction and \emph{acceleration} through the Nesterov's estimating sequences \cite{nesterov2003introductory} is proposed in \cite{kulunchakov2020estimate}. In \cite{taylor2019stochastic}, the authors propose a systematic way of deriving potential-based proofs of the convergence of stochastic first-order methods for stochastic convex smooth optimization problems. Unified approaches to the analysis of variance-reduced and coordinate-wise methods are given in \cite{gower2021stochastic,hanzely2019one}. For a quite general class of composite problems with a non-smooth part having a finite-sum form, the authors of \cite{mishchenko2019stochastic} proposed Stochastic Decoupling Method generalizing many existing stochastic first-order methods including proximal \algname{SGD}, \algname{SAGA}, \algname{Point-SAGA} \cite{defazio2016simple} and many more. There are also several works with generalized analyses of different distributed methods. In particular, unified approaches for the methods involving communication compression are studied in \cite{condat2022murana,condat2022ef,shulgin2022shifted,richtarik20223pc}, and for the decentralized methods with local steps, generalized analyses are given in \cite{koloskova2020unified,huang2022tackling}.

\section{Conclusions}

In this short note, the unified approach to the analysis of \algname{SGD}-type methods from \cite{gorbunov2020unified} was discussed. This analysis is simple and tight, making it very convenient for the first introduction to the convergence of \algname{SGD}-like methods, e.g., one can use this in teaching and tutorials. However, this technique can be useful for research purposes due to the systematization and links between various techniques in stochastic optimization. For example, the technique can simplify the analysis \cite{salim2022dualize} and can lead to the new methods: the first linearly converging methods with error feedback and local steps were obtained with the help of the unified analysis \cite{gorbunov2020linearly,gorbunov2021local}.

It is also worth mentioning several important limitations of the discussed framework. First, it does not cover methods with biased estimators like in \algname{SAG} \cite{schmidt2017minimizing}, \algname{SARAH} \cite{nguyen2017sarah} or \algname{clipped-SGD} \cite{pascanu2013difficulty}. In the context of methods with coordinate-wise randomization, the analysis does not support non-uniform sampling. Next, accelerated stochastic methods like \algname{AC-SA} \cite{ghadimi2012optimal} and \algname{Katyusha} \cite{allen2017katyusha} do not fit the considered framework. Finally, the analysis implicitly relies on the smoothness of the problem (all known special cases use smoothness of $f$) and the boundedness of the variance of the estimator (at least at the given point). However, both assumptions can be relaxed in many ways \cite{zhang2019gradient,zhang2020adaptive,mai2021stability}. It would be interesting to propose a general framework covering at least some of these aspects.

\section{Cross-References}

\begin{itemize}
    \item Monte-Carlo Simulations for Stochastic Optimization
    \item SSC Minimization Algorithms for Nonsmooth and Stochastic Optimization
    \item Stochastic Bilevel Programs
    \item Stochastic Global Optimization: Stopping Rules
    \item Stochastic Programming: Minimax Approach
    \item Stochastic Quasigradient Methods
    \item Stochastic gradient descent
    \item Stochastic Lookahead Optimization with Bayesian Forecasting
    \item Randomized gradient-free methods in optimization and saddle-point problems
\end{itemize}

\section{Acknowledgements}

The author thanks Samuel Horv\'ath for the suggestions for the improvements of the text. The research was partially supported by the Ministry of Science and Higher Education of the Russian Federation (Goszadaniye) No.075-00337-20-03, project No. 0714-2020-0005.

\bibliographystyle{abbrv}
\bibliography{references}

\end{document}